\documentclass[10pt,a4paper]{article}
\usepackage{amsmath, amssymb, amsthm, verbatim} 
\usepackage{graphicx}
\usepackage{pxfonts} 

\usepackage{bm}
\usepackage[all]{xy}
\usepackage{eucal} 
\theoremstyle{plain}
  \newtheorem{thm}{Theorem}[section]
  \newtheorem{prop}[thm]{Proposition}
  
  \newtheorem{cor}[thm]{Corollary}

  \newtheorem{defn}[thm]{Definition}
  \newtheorem{rem}[thm]{Remark}

\def\address#1#2{\begingroup
\noindent\parbox[t]{7.8cm}{%
\small{\scshape\ignorespaces#1}\par\vskip1ex
\noindent\small{\itshape E-mail address}%
\/: #2\par\vskip4ex}\hfill%
\endgroup}%

\begin{document} 
\title{Twisted Poincar{\'e} Lemma and Twisted {\v{C}}ech-de Rham 
Isomorphism \\ 
in case of Projective Line} 
\author{Ko-Ki Ito}   
\maketitle


\section{Introduction} 
As is well known, a system of inhomogeneous linear differential equations 
\begin{align*} 
( d - A ) \begin{bmatrix} g_1 \\ \vdots \\ g _N \end{bmatrix} 
= \begin{bmatrix} \eta _1 \\ \vdots \\ \eta _N \end{bmatrix}  
\end{align*} 
for $n$ $1$-forms $\eta _1 , \ldots , \eta _N$ and 
an $N \times N$-matrix $A$ whose entries are $1$-forms can be solved by the {\it method of variation of constants} 
in case the connection $d-A$ is integrable, 
which is summarized as follows. 
The integrability tells us that 
the sheaf $\mathcal{L}$ of solutions of the 
homogeneous equation 
(that is, $\mathcal{L} = \mathrm{Ker} ( d-A) \subset \mathcal{O} ^{\oplus N}$) 
is a local system of rank $N$, and that the following diagram is commutative: 
\begin{align*} 
\xymatrix{\mathcal{O} ^{\oplus N} \ar[rr] ^{\Phi} _{\cong} \ar[d] ^{d-A} 
&& \mathcal{L} \otimes _{\Bbb{C}} \mathcal{O}    \ar[d] ^{1 \otimes d} \\ 
\mathcal{O} ^{\oplus N} \otimes _{\mathcal{O}} \Omega ^1 
\ar[rr] ^{\Phi} _{\cong} && 
 \mathcal{L}  \otimes _{\Bbb{C}} \Omega ^1   
 . }
\end{align*}   
Thus,  we have $d-A = \Phi ^{-1} \circ (1 \otimes d) \circ \Phi$,  
and a solution is given by 
$\begin{bmatrix} g_1 \\ \vdots \\ g_{N} \end{bmatrix} 
=\Phi \circ \int \circ \Phi ^{-1} 
\left( \begin{bmatrix} \eta _1 \\ \vdots \\ \eta _N \end{bmatrix} \right) $. 
The isomorphism $\Phi$ is locally given by 
$\Phi (  \varsigma _1 f_1 + \cdots +  \varsigma _N f_N ) 
=  \varsigma _1 \otimes f_1+ \cdots +  \varsigma _N \otimes f_N$, 
where $\{ \varsigma _1 , \ldots , \varsigma _N \}$ is a (local) 
basis of $\mathcal{L}$. 
(The matrix corresponding to $\Phi$ is nothing but the {\it Wronski matrix}.) 
In case $N=1$ and $A=-\frac{d \varsigma }{\varsigma}$, 
where $\varsigma = (t-x_0 ) ^{\alpha _0} (t-x_1 ) ^{\alpha _1} \cdots 
(t-x_{n-1} ) ^{\alpha _{n-1}}$ and $t$ is a standard coordinate function on 
$\Bbb{C}$, 
a solution $g$ is given by 
\begin{align} \label{sol1} 
 g = \frac{1}{\varsigma} \int \varsigma \eta . 
\end{align} 
This expression is, however, only valid locally, because $\varsigma$ is a multi-valued function. 
On the domain $U_i$ defined as in figure {\ref{fig1}} 
(which is homotopic to a punctured disc),  
the functions (\ref{sol1}) can be defined single-valued. 
In this paper, 
we shall realize a single-valued solution by finding 
a successful integration path (which has been formulated as a  
{\it regularization of paths} by Aomoto{\cite{aomoto}}.) 
One can call this result the {\it twisted Poincar{\'e} lemma}. 

Using our method, 
we can describe explicitly the isomorphism 
between the twisted de Rham cohomology and 
the {\v{C}}ech cohomology with its coefficients in $\mathcal{L}$. 
A {\v{C}}ech $1$-cocycle $(s_{ij} ) _{ij}$ for the covering 
$\frak{U}  = \{ U_i \}$ is given by 
$s_{ij} = \frac{a_{ij}}{\varsigma} $, where $a_{ij} \in \Bbb{C}$. 
For a twisted de Rham cohomology class $[\eta ]$, 
the corresponding {\v{C}}ech $1$-cocycle is given by 
\begin{align} \label{int_0}
a_{ij} = \int _{reg (i,j)} \varsigma \eta , 
\end{align}     
where $reg(i,j)$ is a regularization of a path connecting $x_i$ to $x_j$. 
This formula (\ref{int_0}) is nothing but a hypergeometric integral. 

Here we encounter with the Wronski matrix. 
The twisted de Rham cohomology is parameterized analytically by 
$( x_0 , \ldots , x_{n-1} ) \in S$, 
where $S$ is a configuration space of $n$-points on $\Bbb{C}$, 
that is, 
$S= \Bbb{C} ^n \setminus \bigcup _{i \not= j} \{ x_i = x_j \}$. 
It forms an analytic vector bundle $\mathcal{H} ^1$, 
which has a natural integrable connection $\nabla$ 
({\it Gau\ss -Manin connection}.) 
On the other hand, 
the {\v{C}}ech cohomology forms another analytic vector bundle 
$\check{\mathcal{H}} ^1$.  
Using a {\it standard} basis  
$\{ e_1 , \ldots , e_{n-1} \}$ of the {\v{C}}ech cohomology, 
we shall trivialize $\check{\mathcal{H}} ^1$ 
to  
$( \Bbb{C} e_1 \oplus \cdots \oplus \Bbb{C} e_{n-1} ) \otimes _{\Bbb{C}} 
\mathcal{O} _{V_{i_0 \ldots i_{n-1}}} 
$ 
on an open set 
$V _{i_0 \ldots i_{n-1}}:= \{ \mathrm{Re} x _{i_0} < \mathrm{Re} x_{i_1} 
< \cdots < \mathrm{Re} x_{i_{n-1}}  \}$.  
The {\v{C}}ech-de Rham isomorphism leads the following commutative 
diagram: 
\begin{align*} 
\xymatrix{
\mathcal{H} ^1 \big| _{V_{i_0 \ldots i_{n-1}}} \ar[rr] ^{\cong} \ar[d] _{\nabla} && 
( \Bbb{C} e_1 \oplus \cdots \oplus \Bbb{C} e_{n-1} ) \otimes _{\Bbb{C}}  
\mathcal{O} _{V_{i_0 \ldots i_{n-1}}} 
\ar[d] ^{1 \otimes d} 
\\ 
\mathcal{H} ^1 \big| _{V_{i_0 \ldots i_{n-1}}} 
\otimes _{\mathcal{O} _{V_{i_0 \ldots i_{n-1}}}} 
\Omega _{V_{i_0 \ldots i_{n-1}}} ^1 
\ar[rr] ^{\cong} && 
( \Bbb{C} e_1 \oplus \cdots \oplus \Bbb{C} e_{n-1} ) \otimes _{\Bbb{C}} 
\Omega _{V_{i_0 \ldots i_{n-1}}} ^1  
} 
\end{align*} 
This result explains that  
the matrix corresponding to the {\v{C}}ech-de Rham isomorphism 
is the Wronski matrix, 
and that 
{\v{C}}ech cocycles 
$\{e_1 ,\ldots , e_{n-1} \}$ 
give  solutions of the hypergeometric system 
$\nabla g=0$.

\section{Twisted Poincar{\'e} lemma} 
For an $n$-tuple 
$x= ( x_0 , x_1 ,\ldots x_{n-1} )$ of points on $\Bbb{C}$, 
let $X_x$ be the punctured projective line: $X_x = \Bbb{P} ^1 \setminus 
\{ x_0 , x_1 , \ldots , x_n = \infty \}$. 
We shall consider a {\it twisted differential}:    
\begin{align} {\label{twisted_diff}} 
 d + \omega 
: \mathcal{O} _{X_x} \longrightarrow \Omega _{X_x} ^1  ,   
\end{align}   
where 
\begin{align} {\label{conn_form}} 
\omega  := \alpha _0 \frac{dt}{t-x_0} 
+ \alpha _1 \frac{dt}{t-x_1} 
+ \cdots 
+ \alpha _{n-1} \frac{dt}{t-x_{n-1}}  
\end{align} 
and we assume that $\alpha _i \not\in \Bbb{Z}$. 
Let $U_i$ be the open set in $X _{x}$ defined by removing 
  $n-1$ lines $\{ l_j \} _{j=0 , 1 , \ldots , n-1 , j \not =i}$ 
from $X_x$,  
where $l_j$ is a path connecting the point $x_j$ and $\infty$;
for instance,   
$l_j := \left\{  t \in \Bbb{C} \ \left| 
\arg ( t - x_j ) = \pi /2 
\right. \right\}$.  
\begin{figure}[h]
\begin{center}  
\unitlength 0.1in
\begin{picture}( 25.7000, 14.3000)(  4.3000,-15.3000)
%
\special{pn 20}%
\special{ar 958 1442 60 60  0.0000000 6.2831853}%
%
\special{pn 20}%
\special{ar 1542 1442 58 58  0.0000000 6.2831853}%
%
\special{pn 20}%
\special{ar 2242 1442 58 58  0.0000000 6.2831853}%
%
\special{pn 20}%
\special{ar 2942 1442 58 58  0.0000000 6.2831853}%
%
\special{pn 20}%
\special{ar 1892 158 58 58  0.0000000 6.2831853}%
%
\special{pn 20}%
\special{pa 958 1384}%
\special{pa 958 800}%
\special{da 0.070}%
\special{pa 958 800}%
\special{pa 958 800}%
\special{da 0.070}%
%
\special{pn 20}%
\special{pa 1542 1384}%
\special{pa 1542 800}%
\special{da 0.070}%
\special{pa 2942 1384}%
\special{pa 2942 800}%
\special{da 0.070}%
%
\special{pn 20}%
\special{pa 958 568}%
\special{pa 970 536}%
\special{pa 984 506}%
\special{pa 996 474}%
\special{pa 1010 444}%
\special{pa 1024 416}%
\special{pa 1038 388}%
\special{pa 1054 362}%
\special{pa 1072 338}%
\special{pa 1092 316}%
\special{pa 1112 294}%
\special{pa 1134 276}%
\special{pa 1158 258}%
\special{pa 1182 242}%
\special{pa 1208 230}%
\special{pa 1236 216}%
\special{pa 1264 206}%
\special{pa 1292 196}%
\special{pa 1322 188}%
\special{pa 1354 180}%
\special{pa 1386 174}%
\special{pa 1420 170}%
\special{pa 1452 166}%
\special{pa 1486 162}%
\special{pa 1522 160}%
\special{pa 1558 158}%
\special{pa 1594 156}%
\special{pa 1630 156}%
\special{pa 1666 156}%
\special{pa 1704 156}%
\special{pa 1740 156}%
\special{pa 1778 158}%
\special{pa 1816 158}%
\special{pa 1834 158}%
\special{sp 0.070}%
%
\special{pn 20}%
\special{pa 2942 568}%
\special{pa 2922 542}%
\special{pa 2900 518}%
\special{pa 2878 494}%
\special{pa 2856 470}%
\special{pa 2834 446}%
\special{pa 2812 422}%
\special{pa 2788 400}%
\special{pa 2766 380}%
\special{pa 2742 358}%
\special{pa 2718 340}%
\special{pa 2692 322}%
\special{pa 2666 306}%
\special{pa 2640 290}%
\special{pa 2612 276}%
\special{pa 2584 262}%
\special{pa 2556 252}%
\special{pa 2526 240}%
\special{pa 2496 230}%
\special{pa 2466 222}%
\special{pa 2436 214}%
\special{pa 2404 206}%
\special{pa 2374 200}%
\special{pa 2342 194}%
\special{pa 2308 188}%
\special{pa 2276 184}%
\special{pa 2244 180}%
\special{pa 2210 176}%
\special{pa 2176 172}%
\special{pa 2144 170}%
\special{pa 2110 168}%
\special{pa 2076 166}%
\special{pa 2042 164}%
\special{pa 2008 162}%
\special{pa 1974 160}%
\special{pa 1940 158}%
\special{pa 1938 158}%
\special{sp 0.070}%
%
\special{pn 20}%
\special{pa 1542 568}%
\special{pa 1554 536}%
\special{pa 1566 506}%
\special{pa 1578 476}%
\special{pa 1590 446}%
\special{pa 1604 418}%
\special{pa 1620 390}%
\special{pa 1636 362}%
\special{pa 1656 338}%
\special{pa 1676 314}%
\special{pa 1698 292}%
\special{pa 1722 270}%
\special{pa 1746 250}%
\special{pa 1772 232}%
\special{pa 1800 212}%
\special{pa 1828 194}%
\special{pa 1846 182}%
\special{sp 0.070}%
\put(8.9000,-15.3000){\makebox(0,0)[lt]{$x_0$}}%
\put(14.9000,-15.3000){\makebox(0,0)[lt]{$x_1$}}%
\put(22.0000,-15.3000){\makebox(0,0)[lt]{$x_i$}}%
\put(29.0000,-15.3000){\makebox(0,0)[lt]{$x_{n-1}$}}%
\put(19.0000,-2.5000){\makebox(0,0)[lt]{$\infty$}}%
\put(8.8000,-8.4000){\makebox(0,0)[rt]{$l_0$}}%
\put(14.8000,-8.4000){\makebox(0,0)[rt]{$l_1$}}%
\put(28.8000,-8.4000){\makebox(0,0)[rt]{$l_{n-1}$}}%
\end{picture}%
\end{center} 
\caption{$U_i$}  \label{fig1} 
\end{figure} 
\begin{thm}[Twisted Poincar{\' e} lemma] \label{thm_P} 
Let $\eta$ be a (single-valued) holomorphic $1$-form on $U_i$.  
Then there exists a (single-valued) holomorphic function $g$ on $U_i$ 
such that 
\begin{align} {\label{diff_eq}} 
(d+\omega ) g = \eta . 
\end{align} 
\end{thm} 
We shall prove this theorem in {\S}{\ref{int}}.

\section{Integrations over twisted chains} \label{int} 
To prove Theorem {\ref{thm_P}}, we shall introduce integrations 
over twisted chains. 

The twisted differential {(\ref{twisted_diff})} 
is locally equal to $\varsigma ^{-1} \circ d \circ \varsigma $, 
where $\varsigma ^{-1}$ is a local solution of the equation 
$(d + \omega ) s =0$. 
Thus a solution $g$ of the equation {(\ref{diff_eq})} 
is locally given by  
\begin{align} \label{local_sol}
g = \varsigma ^{-1} \int  \varsigma \eta .    
\end{align} 
In order for this expression to make sense globally, it is necessary to define  integrations of a multi-valued function. 
A multi-valuedness of $\varsigma$ is controlled by the following local system: 
\begin{align*} 
\mathcal{L} ^{\vee} := \mathrm{Ker} (d-\omega) ,   
\end{align*} 
which is dual to the local system 
$\mathcal{L} := \mathrm{Ker} (d+\omega )$. 
Let $U$ be an open set in $X_x$ defined by removing $n$ lines 
$\{ l_j \} _{j=0 , 1 , \ldots , n-1 }$ 
from $\Bbb{P} ^1$. 
We can take a non-zero section $\varsigma$ of $\mathcal{L} ^{\vee}$ on $U$ 
and fix it once for all. 
To determine sections of $\mathcal{L} ^{\vee}$ on paths in $X_x$, 
we take a point $p \in U$. 
For a path $\gamma$ in $X_x$ whose initial point is $p$, 
we denote by $\varsigma _{\gamma}$ the analytic continuation of $\varsigma$ along 
$\gamma$. 
\begin{defn}[twisted chain] 
 A twisted chain is a chain with its coefficients in $\mathcal{L} ^{\vee}$, 
that is, 
a linear combination of 
$\{ \gamma \otimes s _{\gamma} \} _{\gamma }$, 
where 
$\gamma$ is a singular $1$-simplex (i.e. a path) and 
$s _{\gamma}$ is a local section of $\mathcal{L} ^{\vee}$ on $\gamma$. 
\end{defn} 
\begin{defn}[regularization] \label{regularization} 
Let $\gamma$ be a path on $U_i$ 
whose initial point is $p$. 
The regularization of $\gamma$ is defined by 
\begin{align*} 
reg _i \gamma := \gamma \otimes \varsigma _{\gamma} 
+ \frac{1}{c_i -1} \sigma _i \otimes \varsigma _{\sigma _i} ,  
\end{align*} 
where $c_i := \exp ( 2\pi \sqrt{-1} \alpha _i)$ and  
$\sigma _i$ is a loop around $x_i$ 
whose initial point is $p$. 
\end{defn} 
\begin{defn}[integration over twisted chain] 
Let $\gamma$  be a path whose initial point is $p$.  
The integration over $\gamma \otimes \varsigma _{\gamma}$ 
is defined by 
\begin{align*} 
\int _{\gamma \otimes \varsigma _{\gamma}} \varsigma \eta  
:= \int _{\gamma} \varsigma _{\gamma} \eta . 
\end{align*} 
\end{defn}

\begin{proof}[proof of Theorem {\ref{thm_P}}] 
To construct a global solution of the equation (\ref{diff_eq}), 
we need to continue ({\ref{local_sol}}) analytically on the whole $U_i$. 
In order to achieve this, we consider integrations over regularized paths 
in Definition {\ref{regularization}}. 
\begin{align} \label{int_over_reg}
g (t) := \varsigma _{\gamma _t} ^{-1} \int _{reg _i \gamma _t} \varsigma \eta ,   
\end{align}  
where $\gamma _t$ is a path connecting $p$ to $t$ in $U_i$.   
We shall prove that $g$ is well-defined, that is, 
$g$ is defined independently of a choice of paths $\gamma _t$. 
We take another path $\gamma _t ^{\prime}$. 
It is sufficient to prove 
\begin{align*} 
 \varsigma _{\gamma _t ^{\prime} } ^{-1} \int _{reg _i \gamma _t ^{\prime}} 
\varsigma \eta 
-\varsigma _{\gamma _t} ^{-1} \int _{reg _i \gamma _t} \varsigma \eta = 0
\end{align*} 
in case 
$ \gamma _t ^{-1} \circ \gamma _t ^{\prime} $ 
is homotopic in $U_i$ to $\sigma _i$. 
Note that 
$\varsigma _{\gamma _t ^{\prime}} = c_j \varsigma _{\gamma _t}$.  
\begin{align*} 
\varsigma _{\gamma _t ^{\prime} } ^{-1} \int _{reg _i \gamma _t ^{\prime}} 
\varsigma \eta 
-\varsigma _{\gamma _t} ^{-1} \int _{reg _i \gamma _t} \varsigma \eta 
= \varsigma _{t} ^{-1} 
 \int _{ c_j ^{-1} reg _i \gamma _t ^{\prime} - reg _i \gamma _t } \varsigma \eta , 
\end{align*} 
where 
\begin{align*}  
c_j ^{-1} reg _i \gamma _t ^{\prime} - reg _i \gamma _t 
& = c_j ^{-1} \gamma _t ^{\prime} \otimes \varsigma _{\gamma _t ^{\prime}} 
+ \frac{c_j ^{-1}}{c_j -1} \sigma _i \otimes \varsigma _{\sigma _i} 
- \gamma _t \otimes \varsigma _{\gamma _t} 
- \frac{1}{c_j -1} \sigma _i \otimes \varsigma _{\sigma _i} \\ 
& =  c_j ^{-1} \left( \gamma _t  
\otimes c_j \varsigma _{\gamma _t } 
+\sigma _i \otimes \varsigma _{\sigma _i} \right) 
+ \frac{c_j ^{-1}}{c_j -1} \sigma _i \otimes \varsigma _{\sigma _i} 
- \gamma _t \otimes \varsigma _{\gamma _t} 
- \frac{1}{c_j -1} \sigma _i \otimes \varsigma _{\sigma _i}  \\ 
& = 0 . 
\end{align*} 
We have thus proved the theorem. 
\end{proof}

\section{Twisted {\v{C}}ech-de Rham isomorphism} 
The twisted differential ({\ref{twisted_diff}}) 
defines a {\it{twisted de Rham complex}}: 
\begin{align} \label{dRcomplex} 
0 \longrightarrow 
\Gamma ( X_x , \mathcal{O} _{X_x}) \overset{d+\omega}{\longrightarrow}  
\Gamma ( X_x , \Omega _{X_x} ^1 ) \longrightarrow 
0 .   
\end{align} 
On the other hand, 
the {\it {\v{C}}ech complex with its coefficients in $\mathcal{L}$}
associated to the covering $\frak{U} = \{ U_i \}$ 
given by 
\begin{align} \label{Ccomplex}  
0 \longrightarrow 
\bigoplus _i \Gamma ( U_i ,\mathcal{L} ) 
\overset{\partial ^0}{\longrightarrow} 
\bigoplus _{i<j} \Gamma ( U_i \cap U_j , \mathcal{L}  ) 
\overset{\partial ^1}{\longrightarrow}  
\bigoplus _{i<j<k} \Gamma ( U_i \cap U_j \cap U_k , \mathcal{L} ) 
\longrightarrow \cdots , 
\end{align}  
where 
\begin{align*} 
\left( \partial ^0 ( s_i ) _i \right) _{ij} 
= s_j \big| _{U_i \cap U_j} -s_i \big| _{U_i \cap U_j} , 
\quad 
\left( \partial ^1 ( s_{ij} ) _{ij} \right) _{ijk} 
= s_{jk} \big| _{U_i \cap U_j \cap U_{k}} -s_{ik} \big| _{U_i \cap U_j \cap U_{k}} 
+ s_{ij} \big| _{U_i \cap U_j \cap U_{k}} .     
\end{align*} 
The twisted Poincar{\'{e}} lemma (Theorem {\ref{thm_P}}) 
tells us that 
the first twisted de Rham cohomology 
$H^1 _{d+\omega} ( X_x )$ 
(defined by the complex ({\ref{dRcomplex}}))
is isomorphic to 
the first {\v{C}}ech cohomology 
$H ^1 ( \frak{U} , \mathcal{L} )$ 
(defined by the complex ({\ref{Ccomplex}}).) 
\begin{thm}  \label{Cech_deR_isom}
The morphism 
$\Phi : H^1 _{d+\omega} ( X_x )
\longrightarrow 
 H ^1 ( \frak{U} , \mathcal{L} )$ 
given by 
\begin{align*} 
\Phi ( \eta ) = 
\left( -\frac{1}{\varsigma} \int _{reg(i,j)} \varsigma \eta \right) _{ij}  
\end{align*}  
is well-defined and an injection, 
where 
\begin{align*} 
reg(i,j) =  \frac{1}{c_i -1} \sigma _i \otimes \varsigma _{\sigma _i} 
-\frac{1}{c_j -1} \sigma _j \otimes \varsigma _{\sigma _j} . 
\end{align*}  
\end{thm} 
\begin{rem}  \label{rem1}
The open set $U_{i_0} \cap U_{i_1} \cap \cdots \cap U_{i_k}$ 
coincides with $U$ in case $k >0$. 
Thus 
$\Gamma (U_{i_0} \cap U_{i_1} \cap \cdots \cap U_{i_k} , \mathcal{L} ) 
= \Bbb{C} \frac{1}{\varsigma}$. 
\end{rem} 
\begin{rem} 
Actually, $\Phi$ is an isomorphism. (See Corollary {\ref{cor}}.)
\end{rem} 
\begin{proof}  
We have the following commutative diagram: 
\begin{align*} 
\xymatrix{ 
&& 0  & 0   \\ 
0 \ar[r] & 
\Gamma ( X_x , \Omega _{X_x} ^1 )  \ar[r] & 
\bigoplus _i  \Gamma ( U_i , \Omega _{X_x} ^1 ) \ar[r] ^{\partial _1 ^0} \ar[u] 
&  Z^1 ( \frak{U} , \Omega _{X_x} ^1 )    \ar[u]  \\ 
0 \ar[r] & 
\Gamma ( X_x , \mathcal{O} _{X_x}  ) \ar[r] ^{\iota} \ar[u] & 
\bigoplus _i  \Gamma ( U_i , \mathcal{O} _{X_x}  ) \ar[r] ^{\partial _0 ^0} 
 \ar[u] ^{d+\omega} & 
Z^1 ( \frak{U} , \mathcal{O} _{X_x}  )   \ar[u] ^{d+\omega}  \\ 
&&  \bigoplus _i \Gamma ( U_i , \mathcal{L}  ) \ar[r] \ar[u] ^{\iota ^{\prime}}
&  Z^1 ( \frak{U} , \mathcal{L}   ) \ar[u]  \\ 
&& 0  \ar[u] & 0 \ar[u] 
} 
\end{align*} 
In this diagram, 
both two vertical sequences are exact due to the twisted Poincar{\'{e}} 
lemma (Theorem {\ref{thm_P}} ) 
and 
both two horizontal sequences are exact by definition. 
Here we have 
\begin{align*} 
H^1 _{d+\omega } (X_x) & \cong 
\mathrm{Ker} \partial _1 ^0 
\big/  \mathrm{Im} (d+\omega) \circ \iota   \\ 
H^1 (\frak{U} , \mathcal{L} ) & \cong 
\mathrm{Ker} (d+\omega ) 
\big/  \mathrm{Im} \partial _0 ^0 \circ \iota ^{\prime} 
\end{align*}  
and $\Phi$ should be defined by 
$\partial _0 ^0 \circ (d+\omega ) ^{-1}$. 
A standard argument by diagram chasing tells us that 
$\Phi$ is well-defined and an isomorphism. 
For a $1$-form $\eta$, we  calculate $\Phi (\eta )$ explicitly 
by using formula ({\ref{int_over_reg}}) in the proof of the twisted Poincar{\'{e}} 
lemma: 
\begin{align*} 
\Phi (\eta ) &= 
 \partial _0 ^0 \left( 
\frac{1}{\varsigma _{\gamma _t}} \int _{reg _i \gamma _t} \varsigma \eta 
\right) _i   \\ 
&= 
\left( 
\frac{1}{\varsigma _{\gamma _t}} \int _{reg _j \gamma _t} \varsigma \eta 
- 
\frac{1}{\varsigma _{\gamma _t}} \int _{reg _i \gamma _t} \varsigma \eta 
\right) _{ij} 
\end{align*}  
Note that $\gamma  _t$ is on $U$. 
Thus $\varsigma _{\gamma _t} =\varsigma$ and we have 
\begin{align*} 
reg _j \gamma _t - reg _i \gamma _t &= 
\gamma _t \otimes \varsigma _{\gamma _t} 
+ \frac{1}{c_j -1} \sigma _ \otimes \varsigma _{\sigma _j}  
-\gamma _t \otimes \varsigma _{\gamma _t} 
- \frac{1}{c_i -1} \sigma _i \otimes \varsigma _{\sigma _i} \\ 
&= -reg (i,j) .   
\end{align*} 
We have thus proved the theorem. 
\end{proof}

\section{Explicit formula of twisted {\v{C}}ech-de Rham isomorphism} 
To describe the {\v{C}}ech-de Rham isomorphism $\Phi$ more explicitly 
in matrix form, 
we give generators of cohomology. 
\begin{prop} 
The first cohomology $H ^1 ( \frak{U} , \mathcal{L} )$ of  
the {\v{C}}ech complex  
is generated by $e_1 , \ldots , e_{n-1}$ which are defined by  
\begin{align*} 
e_k := ( s _{ij} ^{(k)} ) _{ij} , \quad 
s_{0i} ^{(k)} =  -\delta _{ik} \frac{1}{\varsigma}, \quad 
s_{ij} ^{(k)} = s_{0j} ^{(k)} -s_{0i} ^{(k)} .  
\end{align*}  
\end{prop} 
\begin{proof} 
Note that we have $\Gamma \left( U_i , \mathcal{L}  \right) =0$, 
because any solution 
$c (t-x_0 )^{-\alpha _0} \cdots (t-x_{n-1} )^{-\alpha _{n-1}}$ 
of $(d+\omega )s=0$ 
cannot be defined as a single-valued function on $U_i$.  
So $H ^1 ( \frak{U} , \mathcal{L} )$ coincides with 
the set $\mathrm{Ker} \partial ^1$ of  
cocycles  
which is denoted by  
$Z ^1 ( \frak{U} , \mathcal{L} )$. 
By Remark {\ref{rem1}}, 
an arbitrary cochain can be expressed by 
$\left( a_{ij} \frac{1}{\varsigma} \right) _{ij}$, where 
$a_{ij} \in \Bbb{C}$. 
Here the cocycle condition is 
$a_{jk} -a_{ik} +a_{ij} =0$. 
Hence $a_{jk} = a_{ik} -a_{ij} =a _{0k} -a_{0j}$. 
\end{proof} 
We introduce the {\it Wronski matrix} $W$ corresponding to $\Phi$. 
\begin{defn}[Wronski matrix] 
We set 
$W:= 
\begin{bmatrix} 
\int _{reg(0,1)} \varsigma \eta _1 & \cdots & \int _{reg(0,1)} \varsigma 
\eta _{n-1} \\ 
\vdots & & \vdots \\ 
\int _{reg(0,n-1)} \varsigma \eta _1 & \cdots & \int _{reg(0,n-1)} \varsigma 
\eta _{n-1} 
\end{bmatrix}$, 
where   $\displaystyle{\eta _k := \frac{dt}{t-x_k}}$.
\end{defn} 
\begin{prop} 
The first cohomology of  twisted de Rham complex 
is generated by $\eta _1 ,\ldots , \eta _{n-1}$.  
\end{prop} 
\begin{proof} 
By Varchenko formula (\ref{varchenko}) ({\cite{varchenko1}}{\cite{varchenko2}}), we have $\det W \not=0$. 
This shows that 
$\{ \Phi (\eta _1 ) , \ldots , \Phi (\eta _{n-1} ) \}$ is linearly independent. 
Hence $\{ \eta _1  , \ldots , \eta _{n-1}  \}$ is linearly independent. 
\end{proof} 
\begin{cor} \label{cor}
The morphism $\Phi$ is an isomorphism and the matrix corresponding to  
$\Phi$ with respect to 
a basis $\{ \eta _1 , \ldots \eta _{n-1} \}$ of $H^1 _{d+\omega} (X_x)$ 
and a basis $\{ e_1 , \ldots , e_{n-1} \}$ of $H^1 (\frak{U} ,\mathcal{L} )$ 
is the Wronski matrix $W$.  
\end{cor} 
\begin{proof} 
This follows immediately from Theorem {\ref{Cech_deR_isom}} and the fact 
$\det W \not=0$.  
\end{proof}

\section{Relative twisted {\v{C}}ech-de Rham isomorphism} 
Recall the punctured projective lines of the form $X _x$ are parametrized by $x$, 
where $x$ runs through the configuration space $S$ 
of $n$-points on $\Bbb{C}$: 
$S := \Bbb{C} ^{n} \setminus \bigcup _{i \not= j} \{ x_i = x_j \}$. 
So the family $\{ X_x \} _{x \in S}$ forms 
an analytic family $\pi : X \longrightarrow S$, 
where $X = 
\left\{ \left. ( t , x ) \in \Bbb{P} ^1 \times S \ \right| \ 
t \not= x_0 , \ldots , x_{n} 
\right\}$.  
(The symbol $x_n$ denotes $\infty$.)    
It induces a vector bundle $\mathcal{H} ^1$ 
over $S$ each of whose fibers is 
$H ^1 _{d+\omega} ( X_x )$. 
Let $DR ^{\bullet} _{d+\omega}$ be a relative de Rham complex 
with twisted differential: 
\begin{align*}  
0 \longrightarrow \mathcal{O} _X 
\overset{d+\omega }{\longrightarrow} 
\Omega ^1 _{X/S} 
\longrightarrow 0 , 
\end{align*}
where $\omega$ is a $1$-form on $X$ defined by the same formula as 
(\ref{conn_form}). 
The vector bundle $\mathcal{H} ^1$ is  
the first cohomology of $\Bbb{R} \pi _{\ast} DR ^{\bullet} _{d+\omega}$. 
Its sections are represented by 
relative $1$-forms, because 
$\pi : X \longrightarrow S$ is Stein. 

The vector bundle $\mathcal{H} ^1$ has a natural connection $\nabla$ 
({\it Gau\ss -Manin connection}). 
For $\left[ g dt \right] \in \mathcal{H} ^1$ represented by 
a relative $1$-form, 
we can write down it explicitly (as seen in \cite{aomotokita}): 
$\nabla \left( [g dt] \right) = 
\left[ \frac{\partial g}{\partial x_0} dt  -\alpha _0 \frac{g dt}{t-x_0} 
 \right] \otimes dx_0 
+ \cdots + 
 \left[ \frac{\partial g}{\partial x_{n-1}} dt -\alpha _{n-1} \frac{g dt}{t-x_{n-1}} 
  \right] \otimes dx_{n-1}$. 

We have another vector bundle $\check{\mathcal{H}} ^1$ 
corresponding to {\v{C}}ech cohomology. 
Let $\mathcal{L} _{X/S}$ be the kernel of 
$d+\omega : \mathcal{O} _X \longrightarrow \Omega _{X/S} ^1$. 
The vector bundle $\check{\mathcal{H}} ^1$ should be defined by 
$R^1 \pi _{\ast} \mathcal{L} _{X/S}$. 
We shall construct and compute 
$R^1 \pi _{\ast} \mathcal{L} _{X}$ by means of {\v{C}}ech resolution. 
Note that $\mathcal{L} _{X/S}$ is isomorphic to 
$\mathcal{L} _{X} \otimes _{\Bbb{C} _{S}} \pi ^{-1} \mathcal{O} _{S}$, 
where $\mathcal{L} _{X}$ is a kernel of 
$d+\omega : \mathcal{O} _X \longrightarrow \Omega_{X} ^1$, 
whose local sections are generated over $\Bbb{C}$ by 
$(t-x_0) ^{-\alpha _0} \cdots (t-x_{n-1} )^{-\alpha _{n-1}}$. 
By projection formula, 
$\check{\mathcal{H}} ^1$ is isomorphic to 
$R^1 \pi _{\ast} \mathcal{L} _X \otimes _{\Bbb{C} _S} \mathcal{O} _S$. 

For $I=(i_0 , \ldots , i_{n-1} )$, we take an open set 
$V _I:= \{ \mathrm{Re} x _{i_0} < \mathrm{Re} x_{i_1} 
< \cdots < \mathrm{Re} x_{i_{n-1}}  \} \subset S$, 
and compute 
$\Gamma ( V _I , R ^1 \pi _{\ast} \mathcal{L} _X )$. 
We have the following {\v{C}}ech resolution: 
\begin{align} \label{r_cech_complex}
0 
\longrightarrow 
\Gamma (  V _I ,  \pi _{\ast} \mathcal{L} _X ) 
\longrightarrow 
\bigoplus _{\alpha} \Gamma ( U_{i_{\alpha}} ^I ,\mathcal{L} _X) 
\overset{\partial ^0}{\longrightarrow} 
\bigoplus _{\alpha < \beta } \Gamma ( U_{i _{\alpha}} ^I \cap 
U_{i_{\beta}} ^I , \mathcal{L}  _X) 
\longrightarrow \cdots , 
\end{align} 
where 
$U _{i _{\alpha}} ^I := \pi ^{-1} ( V _I ) 
\setminus \bigcup _{j \not= i_{\alpha}} 
\left\{ (t,x) \in \Bbb{P} ^1 \times V _I \ \left| 
\ \arg ( t -x_j ) = \pi /2 
\right. 
\right\} 
$.  
\begin{rem} 
Let $U ^I$ be an open set given by 
$\pi ^{-1} ( V _I ) 
\setminus \bigcup _{j } 
\left\{ (t,x) \in \Bbb{P} ^1 \times V _I \ \left| 
\ \arg ( t -x_j ) = \pi /2 
\right. 
\right\} 
$.  
We can regard the function 
$(t-x_0) ^{-\alpha _0} \cdots (t-x_{n-1} )^{-\alpha _{n-1}}$ 
as a single-valued function. 
We denote it by $\frac{1}{\varsigma _I}$. 
Thus $\Gamma ( U_{i _{\alpha}} ^I \cap 
U_{i_{\beta}} ^I , \mathcal{L}  _X) $ is generated by $\frac{1}{\varsigma _I}$ 
because 
$U_{i _{\alpha}} ^I \cap 
U_{i_{\beta}} ^I $ coincides with $U^I$. 
\end{rem} 
\begin{prop} 
The first cohomology of the above complex ({\ref{r_cech_complex}}) 
is generated by 
$e_1 ^I , \ldots , e_{n-1} ^I$ which are defined by  
\begin{align*} 
e_k ^I:= ( s _{ij} ^{(k)} ) _{ij} , \quad 
s_{0i} ^{(k)} =  -\delta _{ik} \frac{1}{\varsigma _I}, \quad 
s_{ij} ^{(k)} = s_{0j} ^{(k)} -s_{0i} ^{(k)} .  
\end{align*}  
\end{prop} 
\begin{thm} 
Let $\eta _{k} := \left[ \frac{dt}{t-x_k} \right] 
\in \Gamma ( V_I , \mathcal{H} ^1 )$ 
be a section represented by a relative $1$-form. 
The module $\mathcal{H} ^1 \big| _{V_I}$ are generated by 
$\eta _1 ,\ldots ,\eta _{n-1}$ over $\mathcal{O} _{V_I}$ 
and we have the following commutative diagram: 
\begin{align*} 
\xymatrix{
 \mathcal{H} ^1 \big| _{V_I} \ar[rr] _{\cong} ^{\Phi _{V_{I}}} 
\ar[d] _{\nabla} && 
( \Bbb{C} e_1 ^I \oplus \cdots \oplus \Bbb{C} e_{n-1} ^I ) 
\otimes _{\Bbb{C} _{V_I}}  
\mathcal{O} _{V_{I}} 
\ar[d] ^{1 \otimes d} 
\\ 
\mathcal{H} ^1 \big| _{V_I} 
\otimes _{\mathcal{O} _{V_{I}}} 
\Omega _{V_{I}} ^1 
\ar[rr] _{\cong} ^{\Phi _{V_I}} && 
( \Bbb{C} e_1 ^I \oplus \cdots \oplus \Bbb{C} e_{n-1} ^I ) 
\otimes _{\Bbb{C} _{V_I}} 
\Omega _{V_{I}} ^1  
} 
\end{align*} 
where the matrix corresponding to $\Phi _{V_I}$ with respect to 
a basis $\{ \eta _1 , \ldots \eta _{n-1} \}$ of $\mathcal{H} ^1 \big| _{V_I}$ 
and a basis $\{ e_1 ^I, \ldots , e_{n-1} ^I \}$  
is the Wronski matrix $W$. 
\end{thm} 
\begin{proof} 
We can prove this theorem in a similar way to the proof of 
Theorem {\ref{Cech_deR_isom}}. 
The commutativity of the diagram is derived from the formula: 
\begin{align*} 
\frac{\partial }{\partial x_j } \int \varsigma gdt  
= \int \frac{\partial }{\partial x_j } ( \varsigma g) dt 
= \int \varsigma \left( \frac{\partial g }{\partial x_j } 
+\frac{1}{\varsigma} \frac{\partial \varsigma}{\partial x_j } g 
\right) dt , 
\end{align*} 
where 
$\displaystyle{\frac{1}{\varsigma} \frac{\partial \varsigma}{\partial x_j } }
= -\alpha _j \frac{1}{t-x_j}$. 
\end{proof}

\appendix
\section{Varchenko formula} 
The determinant of Wronski matrix corresponding to hypergeometric system 
is given by Varchenko (as seen in {\cite{varchenko1}}, {\cite{varchenko2}}). 
The explicit form of it is written by $\Gamma$-factors. 
\begin{prop}[Varchenko formula] 
Let $c (f , \Delta )$ 
be the value of the fixed branch of $f$ 
which is maximum in absolute value on $\Delta$. 
Then we have 
\begin{align} \label{varchenko} 
\det \left[ 
\int _{reg (x _{i-1} , x_{i})} \varsigma \frac{dt}{t-x_j} 
\right] _{1 \leq i \leq N , 1 \leq j \leq N} 
= 
\frac{1}{\alpha _1 \cdots \alpha _{N}} 
\frac{\Gamma (\alpha _0 +1) \cdots \Gamma (\alpha _N +1)}{
\Gamma ( \alpha _0 + \cdots + \alpha _N +1 ) 
} 
\prod _{1 \leq i \leq N, 0 \leq j \leq N} c \left( (t-x_j)^{\alpha _j} , 
reg (x_{i-1} , x_i ) \right) .  
\end{align}   
\end{prop}


\address{Department of Mathematics \\ Faculty of Science \\  Kyoto University \\ 
Kitashirakawa-Oiwakechou, Sakyou \\ Kyoto 606-8502, Japan}{koki@math.kyoto-u.ac.jp}


\begin{thebibliography}{9} 


     
\bibitem{aomoto1} 
    K.Aomoto, 
     Equations aux diff\'erences lin\'eaires et les int\'egrales
              des fonctions multiformes. {I}. {T}h\'eor\`eme d'existence,
   Proc. Japan Acad.,
    50,
     1974,
     413--415




\bibitem{aomoto2} 
    K.Aomoto, 
     \'{E}quations aux diff\'erences lin\'eaires et les
              int\'egrales des fonctions multiformes. {II}.
              \'{E}vanouissement des hypercohomologies et exemples,
   Proc. Japan Acad.,
  50,
  1974,
     542--545


\bibitem{aomoto3}
    K.Aomoto, 
     Les \'equations aux diff\'erences lin\'eaires et les
              int\'egrales des fonctions multiformes,
   J. Fac. Sci. Univ. Tokyo Sect. IA Math.,
    22,
      1975,
    3,
     271--297




  
 
 \bibitem{aomoto4}
    K.Aomoto, 
     On vanishing of cohomology attached to certain many valued
              meromorphic functions,
   J. Math. Soc. Japan,
    27,
     1975,
     248--255,


\bibitem{aomoto}  
    K.Aomoto, 
     On the structure of integrals of power product of linear
              functions,
   Sci. Papers College Gen. Ed. Univ. Tokyo,
    27,
      1977,
    2,
    49--61
  

\bibitem{aomotokita}
      K.Aomoto and M.Kita,
      Choukikakannsuuron (hypergeometric functions) (Japanese),
      Springer-Verlag Tokyo,
      1994   



\bibitem{bott_tu} 
     R.Bott and L.W.Tu, 
     Differential forms in algebraic topology.,  
     Graduate Texts in Mathematics, 82.,
     Springer-Verlag, New York-Berlin, 
     1982. 




\bibitem{deligne} 
      P.Deligne, 
      {\'E}quations diff{\'e}rentielles {\`a}  
                points singuliers r{\'e}guliers. (French) , 
      Lecture Notes in Mathematics , 
      163 , 
      Springer-Verlag,
      1970 


\bibitem{varchenko1} 
      A.N.Varchenko, 
     The Euler Beta-Function, the Vandermonde Determinant, 
     Legendre's Equation, and Critical Values of Linear Functions on 
     a Configuration of Hyperplanes. I , 
     Math. USSR Izvestiya, 35, 1990 , 3 , 543--571

 


\bibitem{varchenko2} 
      A.N.Varchenko, 
     The Euler Beta-Function, the Vandermonde Determinant, 
     Legendre's Equation, and Critical Values of Linear Functions on 
     a Configuration of Hyperplanes. II , 
     Math. USSR Izvestiya, 36, 1991 , 1 , 155--167 


\end{thebibliography}
\end{document}